\documentclass[oneside,reqno,11pt]{amsart}
\usepackage{epsfig}
\usepackage{color}
\usepackage{amssymb,amsmath,amsthm,amstext,amsfonts}
\usepackage[normalem]{ulem}
\usepackage{amsmath,amscd,amstext,amsthm,amsfonts,latexsym}
\usepackage[colorlinks=true,linkcolor=blue, urlcolor=black, citecolor=blue]{hyperref}
\hypersetup{
 citecolor=blue}
 
\makeatletter
\usepackage{tikz}
\newcommand*\circled[1]{\tikz[baseline=(char.base)]{
            \node[shape=circle,draw,inner sep=2pt] (char) {#1};}}
\numberwithin{equation}{section}
\numberwithin{figure}{section}
\theoremstyle{plain}
\newtheorem*{cor*}{\protect\corollaryname}
\theoremstyle{plain}
\newtheorem{thm}{\protect\theoremname}[section]
\theoremstyle{definition}
\newtheorem{defn}[thm]{\protect\definitionname}
\theoremstyle{question}

\theoremstyle{remark}
\newtheorem{rem}[thm]{\protect\remarkname}
\theoremstyle{plain}
\newtheorem{prop}[thm]{\protect\propositionname}
\theoremstyle{plain}
\newtheorem{lem}[thm]{\protect\lemmaname}
\theoremstyle{plain}
\newtheorem{cor}[thm]{\protect\corollaryname}
\newtheorem{ex}[thm]{Example}

\theoremstyle{remark}

\theoremstyle{plain}
\newtheorem{maintheorem}{Theorem}

\newtheorem{maincorollary}[maintheorem]{Corollary}

\theoremstyle{Theorem A}
\theoremstyle{Theorem B}
\theoremstyle{Theorem C}
\theoremstyle{Theorem D}
\theoremstyle{Theorem E}

\DeclareMathOperator{\Leb}{Leb}
\DeclareMathOperator{\ess}{ess}

\numberwithin{equation}{section}

\usepackage[utf8]{inputenc}
\usepackage[T1]{fontenc}
\usepackage{lmodern}

\usepackage{geometry}
\usepackage{amsmath}

\numberwithin{equation}{section}
\numberwithin{figure}{section}
\usepackage{enumitem}		
 \let\footnote=\endnote
\@ifundefined{lettrine}{\usepackage{lettrine}}{}

\theoremstyle{definition}

\def\R{\mathbb{R}}

\def\N{\mathbb{N}}
\def\Z{\mathbb{Z}}

\usepackage{float}


\subjclass[2000]{}

\def\R{\mathbb{R}}

\def\A{\mathcal{A}}
\def\B{\mathcal{B}}

\def\A{\mathcal{A}}

\def\glr{\text{GL}_d(\R)}
\def\gl2{\text{GL}_2(\R)}

\makeatother
\usepackage[T1]{fontenc}
\usepackage[utf8]{inputenc}

  \providecommand{\corollaryname}{Corollary}
  \providecommand{\definitionname}{Definition}
  \providecommand{\lemmaname}{Lemma}
  \providecommand{\propositionname}{Proposition}
  \providecommand{\remarkname}{Remark}
  \providecommand{\theoremname}{Theorem}
\providecommand{\theoremname}{Theorem}

\usepackage{tikz,xcolor,hyperref}

\definecolor{lime}{HTML}{A6CE39}
\DeclareRobustCommand{\orcidicon}{
	\begin{tikzpicture}
	\draw[lime, fill=lime] (0,0) 
	circle [radius=0.16] 
	node[white] {{\fontfamily{qag}\selectfont \tiny ID}};
	\draw[white, fill=white] (-0.0625,0.095) 
	circle [radius=0.007];
	\end{tikzpicture}
	\hspace{-2mm}
}
\foreach \x in {A, ..., Z}{\expandafter\xdef\csname orcid\x\endcsname{\noexpand\href{https://orcid.org/\csname orcidauthor\x\endcsname}
			{\noexpand\orcidicon}}
}

\renewcommand{\le}{\leqslant}
\renewcommand{\leq}{\leqslant}
\renewcommand{\geq}{\geqslant}
\renewcommand{\ge}{\geqslant}

\author{Reza Mohammadpour \orcidA{} }

\address{Reza Mohammadpour, Department of Mathematics, Uppsala University, Box 480, SE-75106, Uppsala, SWEDEN.}
\email{reza.mohammadpour@math.uu.se}

\date{\today}


\begin{document}
\title[Lyapunov exponents and SRB measures for partially hyperbolic systems]{On positive Lyapunov exponents and SRB measures for partially hyperbolic systems}
\maketitle

\begin{abstract}
In this paper we consider $C^{1}$ diffeomorphisms on compact Riemannian
manifolds admitting a dominated splitting $E^{cs} \oplus E^{cu}$. First, we  prove that the smallest Lyapunov exponent along $E^{cu}$, computed with respect to the Lebesgue measure, is computable using observable measures. Then we show that if the Lyapunov exponents along $E^{cu}$ are positive Lebesgue almost everywhere and $E^{cu}$ admits a finest 1-dominated splitting  on the support of an ergodic observable measure then $f$ is non-uniformly expanding along $E^{cu}$. As a byproduct, every $C^{1+\alpha}$ diffeomorphism exhibiting a dominated splitting $E^{s} \oplus E^{cu}$ where $E^{cu}$ fulfills the previous assumptions admits an SRB measure.
\end{abstract}

 
\section{Introduction}

Let 
$(X,T)$ be a dynamical system where 
$X$
is a compact manifold and $T:X \to X$ is a homeomorphism. We denote by $\mathcal{M}(X)$ the space of all Borel probability measures on $X$, and by $\mathcal{M}(X,T)$ the space of all $T$-invariant Borel probability measures, respectively. The space $\mathcal{M}(X)$ is endowed with the weak$^{\ast}$ topology, and $\mathcal{M}(X,T)$ inherits the corresponding subspace topology.

A central question in dynamics is to understand how typical orbits distribute in 
$X$ and which invariant measures describe their long-term behavior. For any point $x \in X$, let $V(x)$ denote the set of accumulation points (in the weak$^{\ast}$ topology) of the sequence
\begin{equation}\label{empirical}
\gamma_{n, x}:=\frac{1}{n} \sum_{j=0}^{n-1} \delta_{T^{j}(x)},
\end{equation}
where $\delta_{y}$ is the Dirac probability measure supported at $y \in X$. The measures $\gamma_{n,x}$ are called the \emph{empirical probability measures} of the orbit of $x$.  

Given an $T$-invariant Borel probability measure $\mu$, the \emph{basin of attraction} of $\mu$ is
\[
B(\mu) = \{ x \in X : V(x) = \{\mu\} \}.
\]
We say that $\mu$ is a \emph{physical measure} for $T$ if its basin of attraction $B(\mu)$ has positive Lebesgue measure in $X$. We denote by $\mathbf{Phys}$ the set of all physical measures of the system.

\begin{defn}
An invariant Borel probability measure $\mu \in \mathcal{M}(X, T)$ is \emph{observable} (or \emph{physical-like}) if, for any $\varepsilon>0$, the set
\[
B_{\varepsilon}(\mu) = \{x \in X: \operatorname{dist}(V(x), \mu)<\varepsilon\}
\]
has positive Lebesgue measure. The set $B_{\varepsilon}(\mu)$ is called the \emph{basin of $\varepsilon$-attraction} of $\mu$. We denote by $\mathcal{O}_T$ the set of all observable measures.
\end{defn}

Observable measures were introduced by Catsigeras and Enrich \cite{CE}, who showed that the set $\mathcal{O}_T$ is non-empty and compact with respect to the weak$^{\ast}$ topology (see \cite[Theorem 1.3]{CE}). Every physical measure is observable, but an observable measure need not be physical. In general, a system may have no physical measures (see Example~\ref{not phy}); however, observable measures always exist (see \cite[Theorem 1.3]{CE}).



Let $f:M \to M$ be a $C^1$ diffeomorphism on a compact Riemannian manifold $M$. Let $\mathcal{A} := Df$ denote the derivative cocycle over $f$. For $x \in M$ and $v \in T_x M$, define the Lyapunov exponent of $v$ at $x$ by
\[
\lambda(x, v, \mathcal{A}) = \limsup_{n \to \infty} \frac{1}{n} \log \| \mathcal{A}^n(x) v \| = \limsup_{n \to \infty} \frac{1}{n} \log \| Df^n_x v \|.
\]

For each $x \in M$, the Lyapunov exponent function $\lambda(x, \cdot, \mathcal{A})$ takes finitely many values
\[
\lambda_1(x, \mathcal{A}) \le \cdots \le \lambda_{p(x)}(x, \mathcal{A}), \quad p(x) \le \dim M,
\]
and these values are invariant along orbits: 
\(\lambda_i(f(x), \mathcal{A}) = \lambda_i(x, \mathcal{A})\) and \(p(f(x)) = p(x)\).

A Borel $f$-invariant measure $\mu$ is \emph{hyperbolic} if, for $\mu$-almost every $x$, all Lyapunov exponents are nonzero and satisfy
\[
\lambda_1(x, \mathcal{A}) \le \cdots \le \lambda_k(x, \mathcal{A}) < 0 < \lambda_{k+1}(x, \mathcal{A}) \le \cdots \le \lambda_{p(x)}(x, \mathcal{A}),
\]
for some integer $k = k(x)$, where the first $k$ exponents are negative and the remaining $p(x)-k$ are positive.

If $\mu$ is ergodic, the Lyapunov exponents are constant almost everywhere, and we write
\[
\lambda_1(\mu, \mathcal{A}) \le \cdots \le \lambda_k(\mu, \mathcal{A}) < 0 < \lambda_{k+1}(\mu, \mathcal{A}) \le \cdots \le \lambda_{p(\mu)}(\mu, \mathcal{A}).
\]

\color{black}
We say that a Borel invariant measure $\mu$ is {\it SRB (Sinai–Ruelle–Bowen)} measure if $\mu$ is hyperbolic and admitting a system of conditional measures such that the conditional measures on unstable manifolds are absolutely continuous with respect to the Lebesgue measures on these manifolds induced by the restriction of the Riemannian structure (see \cite{BDV, Pe, Ya}). We denote by $\textbf{SRB}$ the set of SRB measures and 

In his celebrated ICM talk, M. Viana conjectured that a $C^r$ ($r>1$) diffeomorphism of a compact manifold admits an SRB measure whenever the set of points with positive Lyapunov exponents has positive Lebesgue measure (see \cite{CLP17} for more details).

 In the 1970s, Sinai, Ruelle, and Bowen introduced SRB measures for uniformly hyperbolic attractor sets $\Lambda$ under the assumption that $\Lambda$ has a dense set of periodic points or, equivalently, that $\Lambda$ is locally maximal or has a local product structure. These three properties of uniformly hyperbolic sets are equivalent (cf. \cite{HP}). For partially hyperbolic attractors with the uniform expansion, Pesin and Sinai \cite{PS} developed a technique called the \textit{push-forward} in 1982 to construct special measures known as \textit{$u$-measures}. These measures share many geometric characteristics with the SRB measures. An ergodic SRB measure is physical, but not every physical measure is necessarily an SRB measure. The easiest counterexamples are the point masses on attractive fixed points and periodic orbits. Another interesting example is the figure-eight attractor that has a physical measure with a positive Lyapunov exponent that is not an
SRB measure.

 Alves et al. \cite{ABV} considered the more challenging setting of partially hyperbolic attractors with a continuous splitting $E^{s} \oplus E^{cu}$ with uniform contraction $E^{s}$ and non-uniform expansion $E^{cu}$. In this case, the construction of the SRB measure required significantly more complex techniques due to the non-uniform expansion along $E^{cu}$. In fact, they showed that $f$ admits at least one and finitely many ergodic physical SRB measures \color{black} under the assumption that there exists a set $H$ of positive Lebesgue measure on which $f$ is \textit{non-uniformly expanding} along $E^{cu}$: there exists $\lambda > 0$ such that 
\begin{equation}\label{NUE1}
\limsup_{n\to \infty}\frac{1}{n}\sum_{j=1}^{n}\log \left\| Df_{|E^{cu}(f^{j}(x))}^{-1}\right\| < -\lambda
\qquad \text{for all $x\in H$}.
\end{equation}

In \cite{ADLP}, they improved \cite[Theorem A]{ABV} by showing that the limsup condition \eqref{NUE1} could be replaced by the liminf condition, and they also used Young towers to construct the SRB measure.  We recommend the survey \cite{CLP17, HP} for further references and results.

In this paper, we denote by $\Leb$ a normalized Lebesgue measure with $\Leb(X)=1$. 
Apart from the one-dimensional setting, it has been difficult to work
directly with Lyapunov exponents.  As a result, \cite{ABV, ADLP} introduced stronger versions of non-uniform expansion.  Moreover, \cite{ABV} asked whether their result is true if non-uniformly expanding condition \eqref{NUE1} replaced by
\begin{equation}\label{NUE2}
\limsup_{n\to \infty} \frac{1}{n} \log \left\|Df_{|E^{cu}(f^{n}(x))}^{-n}\right\|<0
\end{equation}
for $\Leb$ almost every $x$, which is the positivity of all Lyapunov exponents in the center unstable direction.

For systems with partial hyperbolicity and non-uniform expansion in the $E^{cu}$ direction, \cite{ABV} proved the existence of invariant probability measures with absolutely continuous conditional measures along the center-unstable direction with respect to Lebesgue measure. These measures are referred to as \textit{Gibbs cu-states} (for a precise definition and properties, see \cite{BDV}).

In this work, we deal with $C^{1}$ diffeomorphisms admitting partially hyperbolic invariant sets, where the tangent bundle over the set has an invariant dominated splitting into two subbundles, namely  non-uniform expansion $E^{cu}$ and non-uniform contraction $E^{cs}$. In this paper, we show that equation \eqref{NUE2} implies equation \eqref{NUE1} under the assumption that the cocycle $Df_{|E^{cu}(f)}^{-1}$ has a 1-dominated splitting on the support of an ergodic observable measure. Our key assumption is that the cocycle $Df_{|E^{cu}(f)}^{-1}$ has a 1-dominated splitting on the support of an ergodic observable measure (the measure always exists, cf. \eqref{existence of obs} and Definition \ref{main-def}) and not on the entire manifold or the attractor. 
Our goal is to show that for partially hyperbolic systems, the positive Lyapunov exponents in the center-unstable direction imply non-uniform expansion under a mild assumption.

\medskip

This paper is organized as follows.
In Subsection~\ref{main results}, we state the main results. Section~\ref{sec:prelim} is devoted to some preliminaries, on observable and physical measures, matrix cocycles, dominated splitting and the Proof of Corollary~\ref{corCinfty}. Finally, in Section~\ref{proofs} we prove the main results.

\subsection{Main results}\label{main results}

\medskip

\subsubsection{Lyapunov ergodic optimization}
Our first main result concerns ergodic optimization for Lyapunov exponents of diffeomorphisms exhibiting a dominated splitting as above. { Let $f: M \rightarrow M$ be a $C^{1}$ diffeomorphism on a compact Riemannian manifold that admits a dominated splitting $T M=E \oplus F$.  Given a non-empty $f$-invariant compact set $Y\subset M$, we say that $Df_{|F}^{-1}$ has a 1-dominated splitting  on $Y$ if there exists a one-dimensional invariant cone fields $(C_x)_{x\in Y}$ on $Y$ for $Df_{|F(f)}^{-1}$. 
Given a $Df$-invariant subbundle $F$, the \textit{Lyapunov exponent} of $Df^{-1}_{|F}$ with respect to an invariant measure $\mu$ and $f$ is defined by \begin{equation}\label{LE-PHS}
    \lambda(\mu, Df^{-1}_{|F}):=\lim_{n \to \infty} \frac{1}{n} \int \log \left\|\left.D f^{-n}\right|_{F\left(f^n x\right)}\right\| d\mu(x).
\end{equation}We recall that $\mathcal{O}_f$ is non-empty and  compact with respect to the weak* topology (see Subsection \ref{subse:Obs} for more details). As the function $\mathcal O_{f} \ni \mu\mapsto \lambda (\mu, Df^{-1}_{|F})$ is upper semi-continuous with respect to the weak* topology (see \cite[Proposition A.1]{FH}) and  $\mathcal O_{f}$ is a compact metric space (see \cite[Theorem 1.3]{CE}), one has
\begin{equation}\label{existence of obs0}
m(Df^{-1}_{|F}):=\sup_{\mu \in \mathcal{O}_{f}}\lambda(\mu, Df^{-1}_{|F})=\max_{\mu \in \mathcal{O}_{f}} \lambda(\mu, Df^{-1}_{|F}).    
\end{equation}

\begin{defn}\label{main-def}
    We say that an $f$-invariant observable measure
    $\mu \in \mathcal{O}_f$ is a maximizing observable measure if 
    $\lambda(\mu, Df^{-1}_{|F}) = m(Df^{-1}_{|F})$.
\end{defn}}
By \eqref{existence of obs0}, a maximizing observable measure always exists. In \eqref{LE-PHS}, the limit concerns Lyapunov exponents for invariant measures, whereas the Lyapunov exponent defined in \eqref{NUE2} is more difficult to handle due to the possible non-invariance of the Lebesgue measure. In the next theorem, we not only relate the Lyapunov exponents given by \eqref{NUE2} to those for observable measures, but also conclude, in particular, that there exists an (invariant) observable measure that realizes such Lyapunov exponents. More precisely:

\begin{maintheorem}\label{mainresult3}
Let $f: M \rightarrow M$ be a $C^{1}$ diffeomorphism on a compact Riemannian manifold that admits a dominated splitting $T M=E \oplus F$.  
 Assume $Df^{-1}_{|F}$ has a 1-dominated splitting on the support of an ergodic maximizing observable measure $\mu \in \mathcal{O}_{f}$.
Then,
\[
\ess \sup \limsup_{n\rightarrow \infty} \frac{1}{n}\log \left\|\left.D f^{-n}\right|_{F\left(f^n x\right)}\right\|=\sup_{\mu \in \mathcal{O}_{f}} \lambda(\mu,  Df^{-1}_{|F})=\max_{\mu \in \mathcal{O}_{f}} \lambda(\mu,  Df^{-1}_{|F}),\]
where the $\ess \sup$ denotes the essential supremum with respect to the Lebesgue measure $\Leb$.
\end{maintheorem}
We recall that a maximizing observable measure $\mu \in \mathcal{O}_{f}$, which is stated in the above theorem, always exists (see \ref{existence of obs0} and Definition \ref{main-def}). The result also applies to diffeomorphisms with partially hyperbolic attractors by restricting the assumptions on the attractors (see Subsection \ref{PH-dominated}).

 Let $f:M \to M$ be a $C^{\infty}$ map on a compact smooth manifold $M$. Kozlovski \cite{K} showed the following integral formula for  the topological entropy of \color{black} $C^{\infty}$ smooth systems:
\begin{equation}\label{definition of entropy}
    h_{\operatorname{top}}(f)=\lim _{n\rightarrow \infty} \frac{1}{n} \log \int \max _{k}\left\|\Lambda^{k} Df_x^{n}\right\| d \operatorname{Leb}(x),
\end{equation}

 where $\Lambda^{k} Df_x^{n}$ is a mapping between $k$-exterior algebras of the tangent spaces. The methods used in the proof of Theorem~\ref{mainresult3} can be used to obtain the following consequence (see (1.3) in \cite{Bu} for a similar result).

\begin{maincorollary}\label{corCinfty}
Assume that $f:M \to M$ is a $C^{\infty}$ map on a compact smooth manifold  $M$. Then
$$
\begin{aligned}
\ess \sup \limsup_{n\rightarrow \infty}\frac{1}{n} \log \max_{k}\|\wedge^{k} Df_x^{n}\|&\leq h_{\operatorname{top}}(f)\\
& \leq 
 \limsup_{n\rightarrow \infty}\frac{1}{n}\log \ess \sup \max_{k}\|\wedge^{k} Df_x^{n}\|,
\end{aligned}
$$
where the $\ess \sup$ denotes the essential supremum with respect to the Lebesgue measure $\Leb$.
\end{maincorollary}

\subsubsection{SRB measures for partially hyperbolic systems}
Assume that 
$f: M \rightarrow M$ is a $C^{1}$ diffeomorphism on a compact Riemannian  manifold that admits a dominated splitting $T M=E^{cs} \oplus E^{cu}$ (see Subsection \ref{PH-dominated} for the definitions).  
Given an $f$-invariant probability measure $\mu$ we denote by 
$$
\lambda^- (\mu, f):=\lambda^- (\mu, Df_{|E^{cu}(f)}^{-1}):=\lim_{n\to \infty} \frac{1}{n} \int \log \left\|D f_{\mid E^{c u}\left(f^n (x)\right)}^{-n}\right\| d\mu(x)
$$
 the minimal Lyapunov exponent of $\mu$ in the center-unstable direction. As the function $\mathcal O_{f} \ni \mu\mapsto \lambda^- (\mu, f)$ is upper semi-continuous with respect to the weak* topology (see \cite[Proposition A.1]{FH}) and  $\mathcal O_{f}$ is a compact metric space (see \cite[Theorem 1.3]{CE}), one has
\begin{equation}\label{existence of obs}
m(f):=\sup_{\mu \in \mathcal{O}_{f}} \lambda^- (\mu, f)=\max_{\mu \in \mathcal{O}_{f}} \lambda^- (\mu, f).    
\end{equation}

\begin{defn}
    We say that an $f$-invariant observable measure
    $\mu \in \mathcal{O}_f$ is a maximizing observable measure \color{black} if 
    $\lambda^- (\mu, f) = m(f)$.
\end{defn}
By \eqref{existence of obs}, a maximizing observable measure always exists.
 Given a non-empty $f$-invariant compact set $Y\subset M$, we say that $Df_{|E^{cu}(f)}^{-1}$ has a 1-dominated splitting  on $Y$ if there exists a one-dimensional invariant cone fields $(C_x)_{x\in Y}$ on $Y$ for $Df_{|E^{cu}(f)}^{-1}$. 
We can now state the main results of the paper, concerning the existence of SRB measures for partially hyperbolic diffeomorphisms,
which will follow as consequences of a more general result for matrix cocycle (cf. Theorem \ref{main-conj}). More precisely, the following theorem addresses the relationship between the Lyapunov exponent in \eqref{NUE2} and the non-uniformly expanding \eqref{NUE1}, and, as a result, establishes the existence of an SRB measure.


\begin{maintheorem}\label{mainresult1}
(Partially hyperbolic diffeomorphisms)
Let $f: M \rightarrow M$ be a $C^{1}$ diffeomorphism on a compact Riemannian  manifold that admits a dominated splitting $T M=E^{cs} \oplus E^{cu}$.  
Assume that
    $Df_{|E^{cu}(f)}^{-1}$ has a 1-dominated splitting on the support of an ergodic maximizing observable measure $\mu \in \mathcal{O}_{f}$. The following properties hold:
    \begin{enumerate}
        \item If 
 \begin{equation}\label{LE00}
   \limsup_{n\to \infty}\frac{1}{n}\log \left\| Df^{-n}_{|E^{cu}(f^{n}(x))}\right\| <0 \qquad \text{Lebesgue a.e. $x$}
\end{equation} 
then there exist $\lambda>0$ and $K \in \mathbb{N}$ such that 
\begin{equation*}\label{NUE}
\limsup _{n \rightarrow \infty} \frac{1}{n K} \sum_{i=1}^{n} \log \left\|D f_{|E^{cu}(f^{i K}(x))}^{-K}\right\| \leq-\lambda<0,
\end{equation*}
for Lebesgue almost every $x$.
        \item  If we replace $C^1$ with $C^{1+\alpha}$ for $f$ and $E^{cs}$ with $E^s$ (uniformly contracting) in the conditions of the theorem, and ~\eqref{LE00} holds, then $f$ has finitely many ergodic physical SRB measures, and the union of their basins has full Lebesgue measure.
    \end{enumerate}
 \end{maintheorem}

\medskip
 SRB measures for such partially hyperbolic diffeomorphisms have been partially studied 
 (see \cite{Burguet, Yongxia, CYZ} and references therein). We observe that
 it is clear that the two formulations 
 of Lyapunov exponents in the previous statement are equivalent when the non-uniformly expanding direction is 1-dimensional, i.e. $E^{cu}=E^{c} \oplus E^{u}$ with $\dim E^{c}=1$ and $E^{u}$ is uniformly expanding.
 Theorem~\ref{mainresult1} is an extension of the latter in the sense that one only requires the domination property on the support of an observable measure (such measures always exist, cf. \eqref{existence of obs}) and not on the attractor.

Finally, it is not hard to deduce a version of Theorem~\ref{mainresult1} for invariant sets of positive Lebesgue measure and observable measures relative to center-unstable disks, as defined in \cite{CYZ}. We leave this as an easy exercise to the reader.

\begin{maincorollary}
    \label{equal O=SRB}
Let $f: M \rightarrow M$ be a $C^{1+\alpha}$ diffeomorphism on a compact Riemannian manifold that admits a dominated splitting $TM=E^s \oplus E^{cu}$, where $E^s$ is uniformly contracting.
If
    $Df_{|E^{cu}(f)}^{-1}$ has a 1-dominated splitting on the support of an ergodic maximizing observable measure $\mu \in \mathcal{O}_{f},$
then
$$
\begin{aligned}
\ess \sup \limsup_{n\to \infty}\frac{1}{n}\log \left\| Df^{-n}_{|E^{cu}(f^{n}(x))}\right\| &=\sup_{\mu \in \mathcal{O}_{f}} \lambda^- (\mu, f)\\
&=\sup_{\mu \in \textbf{Phys}} \lambda^- (\mu, f)\\
&=\sup_{\mu \in \textbf{SRB}} \lambda^- (\mu, f),
\end{aligned}
$$
where the essential supremum is taken with respect to the Lebesgue measure.
\end{maincorollary}


\begin{rem}
The assumption of uniform contraction in the previous corollary can be somewhat relaxed, as explained in Section 6 of \cite{ABV}. 
Let $f: M \rightarrow M$ be a $C^{1+\alpha}$ diffeomorphism on a compact Riemannian  manifold that admits a dominated splitting $T M=E^{cs} \oplus E^{u}$ with $E^{cs}$ mostly contracting, meaning that for any local unstable manifold $\gamma^u$, then we have $\lambda^{cs}(x)<0$
for a positive $\Leb_{\gamma^{u}}$ measure set of points $x\in \gamma^{u},$ where the largest Lyapunov exponent in the $E^{cs}$-direction
  \[ \lambda^{cs}(x):=\limsup_{n\to \infty}\frac{1}{n}\log \|Df_{|E^{cs}(x)}^{n}(x)\|.\]
 Bonatti and Viana \cite{BV} showed that $f$ has finitely many ergodic physical SRB measures and the union of their basins has full Lebesgue measure. In consequence, Corollary \ref{equal O=SRB} can admit a formulation in such context. 
    \end{rem}

\begin{rem}
All results of the present paper also work for diffeomorphisms with partially hyperbolic attractors (see Subsection \ref{PH-dominated}).
\end{rem}

\section{Applications}\label{examples}
In the following section, we will present  several examples which enlighten the concept of observable measures and possible applications of our results.


\begin{ex}\label{main-example1}
Let $f:M \to M$ be a $C^{1}$-diffeomorphism on a smooth compact Riemannian manifold $M$ that admits a dominated splitting $T M=E^{cs} \oplus E^{cu}$. 
Let $\mu \in \mathcal{O}_f$ be an ergodic maximizing observable measure. Assume that $D f_{\mid E^{c u}\left(f\right)}^{-1}$ satisfies the uniform 1-gap property with respect to $\mu$, i.e., there are a $C^0$-neighborhood $\mathcal{U}_{D f_{\mid E^{c u}\left(f\right)}^{-1}}$ of $D f_{\mid E^{c u}\left(f\right)}^{-1}$ 
and $\beta>0$ such that for every $B \in \mathcal{U}_{D f_{\mid E^{c u}\left(f\right)}^{-1}}$ we have
$$
\int\left(\chi_1(x, B)-\chi_2(x, B)\right) d \mu(x)>\beta ,
$$
where 
\(\chi_1(x, B)\) and \(\chi_2(x, B)\) are the largest and second-largest Lyapunov exponents for the cocycle \(\B\).  Then, \cite{Ye} ensures that the cocycle $D f_{\mid E^{c u}\left(f\right)}^{-1}$ has a 1-dominated splitting on the support of the ergodic maximizing observable measure $\mu$. Therefore, by Theorem \ref{mainresult3},
\[
\ess \sup \limsup_{n\rightarrow \infty} \frac{1}{n}\log \left\|\left.D f^{-n}\right|_{E^{c u}\left(f^n x\right)}\right\|=\sup_{\mu \in \mathcal{O}_{f}} \lambda(\mu,  Df^{-1}_{|E^{c u}})=\max_{\mu \in \mathcal{O}_{f}} \lambda(\mu,  Df^{-1}_{|E^{c u}}),\]
where the $\ess \sup$ denotes the essential supremum with respect to the Lebesgue measure $\Leb$. Also, by our Theorem \ref{mainresult1}, there exist $\lambda>0$ and $K \in \mathbb{N}$ such that 
$$
\limsup _{n \rightarrow+\infty} \frac{1}{n K} \sum_{i=1}^{n} \log \left\|D f_{|E^{cu}(f^{i K}(x))}^{-K}\right\| \leq-\lambda<0.
$$
for Lebesgue almost every $x$ if $\ess \sup \limsup_{n\rightarrow \infty} \frac{1}{n}\log \left\|\left.D f^{-n}\right|_{E^{c u}\left(f^n x\right)}\right\|<0.$

Moreover,  if we replace $C^1$ with $C^{1+\alpha}$ for $f$ and $E^{cs}$ with $E^c$ in the conditions of the current example, then $f$ has finitely many ergodic physical SRB 
measures and the union of their basins has full Lebesgue measure by our Theorem \ref{mainresult1}.
\end{ex}

\begin{ex}
Let $f:\mathbb{T}^2 \to \mathbb{T}^2$ be a $C^{1+\alpha}$ Anosov diffeomorphism. We know that there exists a unique SRB measure $\mu_f$ for $f$. 
Let $g: S^1 \to S^1$ be a \emph{North--South diffeomorphism}, that is, a diffeomorphism with exactly two fixed points: $p_1$, the North Pole, which is repelling, and $p_2$, the South Pole, which is attracting.
Suppose that 
$$
\|Df_{|E^s} \| < m(Dg) \leq \|Dg \| < \| Df_{|E^u}\|,
$$ 
where $m(Dg):=\|Dg^{-1}\|^{-1}$. Then $f \times g : \mathbb{T}^3 \to \mathbb{T}^3$ is a partially hyperbolic diffeomorphism. 
The measures $\nu_1:=\mu_f \times \delta_{p_2}$ and $\nu_2:=\mu_f \times \delta_{p_1}$ are observable measures (see Theorem \ref{total}). Moreover, the measure $\nu_1$ is a $u$-Gibbs and physical measure, while the measure $\nu_2$ is a $u$-Gibbs measure, but not physical. 
It is simple to check that $Df_{|E^{cu}(f)}^{-1}$ has a 1-dominated splitting on the support of $\nu_1$.
\end{ex}
\color{black}

The following example relates to observable measures.
\begin{ex}[{Bowen's example}]
We consider the time-1 map of the flow as shown in Figure \ref{fig1}. 
\begin{figure}[H]
    \centering
    \includegraphics[width=0.6\linewidth]{Bowen.png}
    \caption{Bowen's example.}
    \label{fig1}
\end{figure}
The Dirac measure at the hyperbolic fixed point \( p \) is a hyperbolic physical measure whose basin of attraction includes all points except \( q_1 \) and \( q_2 \). By slowing down the flow near \( p \), one can adapt this example so that \( p \) becomes an indifferent fixed point, and hence, the physical measure is no longer hyperbolic, but it is an observable measure.
\end{ex}

\begin{ex}\label{not phy}
Let $g:\mathbb{T}^{2} \to \mathbb{T}^2$ be a transitive $C^2$ Anosov diffeomorphism. By Sinai's theorem,  there exists a $g$-ergodic physical measure $\mu$ on the two-torus, which is an SRB measure for $g$. Denote by $f: \mathbb{T}^3 \to \mathbb{T}^3$ the $C^2$ map on the three-dimensional torus $\mathbb{T}^3$, defined by $f(x, y, z) = (x, g(y, z)).$ For Lebesgue almost all initial states $(x, y, z)\in \mathbb{T}^3$,
the sequence \eqref{empirical} of the empirical probability measures converges to a measure
$\mu_x = \delta_x \times \mu$, which is supported on a 1-dimensional unstable manifold injectively immersed in the two-torus $\{x\} \times \mathbb{T}^2$. Measures $\mu_x$ are mutually singular for different values of $x\in S^1$ as they are supported on disjoint compact two-tori embedded on $\mathbb{T}^2$. The basin of attraction $B(\mu_x)$ has zero Lebesgue measure in the ambient manifold, therefore, none of the probability measures $\mu_x$ is physical for $f$. Moreover, the set of all those measures $\mu_x$, which is weak* compact, coincides with the set $\mathcal{O}_T$ of observable measures for $f$ by Theorem \ref{total}.

\end{ex}



\section{Preliminaries}\label{sec:prelim}
\subsection{Observable measures}\label{subse:Obs}

In this section, we recall the notions and facts related with observable measures.
Let $T:X \to X$ be a homeomorphism on a compact manifold $X$ and let $\Leb$ be a Lebesgue measure normalized so that $\Leb(X)= 1.$

A $T$-invariant Borel probability measure $\mu \in \mathcal{M}(X, T)$ is  observable or physical-like if for any $\varepsilon>0$ the set $B_{\epsilon}(\mu)=\{x \in X: \operatorname{dist}(V(x), \mu)<\varepsilon\}$
has positive Lebesgue measure. We call $B_{\epsilon}(\mu)$ the  basin of $\varepsilon$-attraction of $\mu$.  We denote by $\mathcal{O}_T$ the set of all observable measures.

\begin{rem}\label{approaching} By the definition of observable measures, $\epsilon$-approximation in the space $\mathcal{M}(X, T)$ can be transferred to an $\epsilon$-approximation (in time-mean) towards an
attractor in the ambient manifold $X$. In other words, if $\mu$ is observable and $x\in B_{\epsilon}(\mu)$, then the iterates $T^{n}(x)$ will $\epsilon$-approach the support of $\mu$ with a frequency which is
asymptotically bounded away from zero.
\end{rem}

We recall that
$
B(\mu) = \{ x \in X : V(x) = \{\mu\} \}
$
is the \emph{basin of attraction} of the physical measure $\mu$. \color{black} Inspired by the latter definition, we define {\it the basin of attraction} $\mathcal{O}_T$:
\begin{equation}\label{basin-of-attraction}
\mathcal{G}:= \mathcal{G}(\mathcal{O}_T)=\{x \in X : V(x) \subset \mathcal{O}_T \}.
\end{equation}
Catsigeras and Enrich \cite[Theorems 1.3]{CE} have shown that $\mathcal{O}_T$ is non-empty. Let us collect some other relevant properties from \cite[Theorems 1.5--1.7]{CE}:

\begin{thm}[\cite{CE}]\label{total}
The following properties hold:
\begin{enumerate}
    \item $\mathcal{O}_T$ is the smallest weak* compact set that contains, for Lebesgue almost all initial states, the limits of all convergent subsequences of (\ref{empirical});
    \item There are finitely many physical measures such that the union of their basins of attraction covers $X$  Lebesgue almost everywhere if and only if $\mathcal{O}_T$ is finite and, if this is the case, 
    $\mathcal{O}_T=$\textbf{Phys};
    \item If $\mathcal{O}_T$ is countably infinite, then there are countably infinitely many physical measures such that their basins of attraction cover $X$ Lebesgue almost everywhere and, if this is the case, $\mathcal{O}_T$ is the weak* closure of the set of physical measures.
\end{enumerate}
\end{thm}

Note that item (1) above means that $\mathcal{O}_T=\bigcap_{K \in W} K$, where
\[W:=\{K \subset \mathcal{M}(X): K\text{ is compact and } \Leb(\mathcal{G}(K))=1\}.\]
Moreover, 
it is worth noticing that, by Theorem \ref{total}, $\Leb(\mathcal{G})=1$.


\subsection{Matrix cocycles}\label{matrix cocycles}

A \textit{matrix cocycle} $\A$ over  a topological dynamical system $(X, T)$ is a continuous map $\A \colon X \to \glr$. For $n\in \N$ and $x \in X$, we define the product of $\A$ along the length $n$ orbit of $X$ as   
\begin{equation}\label{product}
\mathcal{A}^{n}(x)=\mathcal{A}(T^{n-1}(x))\mathcal{A}(T^{n-2}(x))\cdots \mathcal{A}(x).
\end{equation}
 When the context is clear, we say that $\A$ is a cocycle. Also, it induces a skew-product 
 $$
 F:X \times \R^{d} \rightarrow X\times \R^{d}\qquad \text{given by} \qquad F(x, v)=(T(x), \mathcal{A}(x)v).
 $$
 Observe that $F^{n}(x, v)=(T^{n}(x), \mathcal{A}^{n}(x)v)$ for each $n \geq 1$.
In case 
$T$ is invertible then so is $F$ and $F^{-n}(x)=(T^{-n}(x), \mathcal{A}^{-n}(x)v)$ for each $n\geq1$, where
\[\mathcal{A}^{-n}(x):=\mathcal{A}(T^{-n}(x))^{-1}\mathcal{A}(T^{-n+1}(x))^{-1} \dots \, \mathcal{A}(T^{-1}(x))^{-1}.\]
Using the submultiplicativity of the norms of matrices, Kingman's sub-additive ergodic theorem ensures that
if $\mu \in \mathcal{M}(X, T)$ and
$\log ^{+} \|\mathcal A^{\pm1}\| \in L^{1}(\mu)$ then 
the limit 
$$
\chi(x, \mathcal{A}):=\lim _{n \rightarrow \infty} \frac{1}{n} \log \|\mathcal A^n(x)\|,
$$
is well defined 
$\mu$-almost everywhere, and it is called the \textit{top Lyapunov exponent} at $x$.

\begin{rem}
An important class of matrix cocycles are \textit{derivative cocycles}, which are defined as follows. Let $f: M \rightarrow M$ be a diffeomorphism on a manifold $M$. Assume that the tangent bundle $T M$ is trivial. We then consider $\A(x)=D f_x$ at each point as a matrix. Thus the tangent map
$$
D f: T M \rightarrow T M, \quad D f(x, v)=(f(x), \A(x) v)
$$
is a skew product that generated by a matrix cocycle $(f, \A)$, in the previous sense. 
\end{rem}

\begin{thm}[{{Oseledets}}]Let $\A: X\to \glr$ be a matrix cocycle over  a homeomorphism $(X, T)$ and $Gr(d)$ denote the Grassmanian of $\R^d.$ Let $\mu$ be an $T$-invariant measure. For $\mu$ almost every $x\in X$, there are functions $k=k(x)$, $\lambda_1(x, \mathcal{A}) > \lambda_2(x, \mathcal{A}) > \ldots > \lambda_k(x, \mathcal{A})$ and a direct sum decomposition $\R^d = E^1(x) \oplus E^2(x) \oplus \ldots E^k(x)$ such that, for every $i=1, \ldots, k:$ 
\begin{itemize}
\item[1)] The maps $x \to k(x)$, $x \to \lambda(x, \mathcal{A})$ and $x \to E^i(x)$ (with values $\N$, $\R$ and $Gr(d)$, respectively) are measurable.
\item[2)] $k(T(x)) = k(x)$, $\lambda_{i}(T(x), \mathcal{A}) = \lambda_{i}(x, \mathcal{A})$ and $\mathcal{A}(x)E^i(x) = E^i(T(x)).$

\item[3)] For every non-null $v\in E^i(x)$,
\[\lim_{n \to \pm \infty}\frac{1}{n}\log \|\mathcal{A}^{n}(x)v\|=\lambda_{i}(x, \mathcal{A}).\]
\item[4)] For any disjoint sum of the spaces within the splitting
\[\lim_{n\to \pm \infty}\frac{1}{n} \log |\sin \measuredangle (\oplus_{i \in I}E^i(T^n(x)),\oplus_{i \in J}E^i(T^n(x))|= 0\] whenever $I \cap J= \emptyset$.
\end{itemize}
\end{thm}

The splitting $\R^d = E^1(x)\oplus \ldots E^k(x)$ is called the {\it Oseledets splitting} at $x$. 
We denote by $\chi_1(x, \mathcal{A}) \geq \chi_2(x, \mathcal{A}) > \ldots \geq \chi_{d}(x, \mathcal{A})$ the \textit{Lyapunov exponents} with respect to $\mu$, counted with multiplicity (the multiplicity of $\lambda_{i}(x, \mathcal{A})$ is the dimension of $E^i(x)$). This is the \textit{Lyapunov spectrum} of the cocycle with respect to some invariant measure $\mu$.
Let  $\chi(\mu, \A) := \int \chi_1(. , \mathcal{A} ) d\mu$. If the measure $\mu$ is ergodic, then $\chi_1(x, \mathcal{A}) = \chi(\mu, \A)$ for $\mu$-almost every $x\in X$.


 \subsection{Dominated splitting}\label{PH-dominated}

Let $f$ belong to the space $\operatorname{Diff}^r(M)$ of $C^r$-diffeomorphisms endowed with the $C^r$-topology, where $r \geq 1$ and $M$ is a closed connected $d$-dimensional Riemannian manifold. For $K \subset M$, a \textit{splitting} of $T_K M = E^1 \oplus E^2 \oplus \ldots \oplus E^{\ell}$ is a linear decomposition $T_x M = E^1(x) \oplus E^2(x) \oplus \ldots \oplus E^{\ell}(x)$ at each $x \in K$ such that $\operatorname{dim} E^i(x)$ does not depend on $x$ for any $1 \leq i \leq \ell$.

\begin{defn}
Let $K \subset M$ be an $f$-invariant set (i.e., $f(K) = K$). We say that a splitting $T_K M = E^1 \oplus E^2 \oplus \ldots \oplus E^{\ell}$ is a \textit{dominated splitting} for the cocycle $Df$ if and only if the following conditions hold:
\begin{itemize}
\item ($Df$-invariance) for every $x \in K$ and $1 \leq i \leq \ell$, we have $Df_x(E^i(x)) = E^i(f(x))$.
\item (domination) there exist constants $c > 0$ and $\lambda \in (0,1)$ such that for every $1 \leq i \leq \ell-1$, every $x \in K$, and every nonzero vectors $u \in E^i(x)$ and $v \in E^{i+1}(x)$, the following inequality holds:
$$
\frac{\|Df_x^n u\|}{\|u\|} \leq c \lambda^n \frac{\|Df_x^n v\|}{\|v\|} \quad \text{for all } n \geq 0.
$$
\end{itemize}    
\end{defn}

\begin{defn}
     For a compact invariant set $K$, a $D f$-invariant bundle $F$ is:
     \begin{enumerate}
         \item \textit{uniformly contracted} (by $D f)$ if there are constants $C>0$ and $\lambda \in(0,1)$ such that for any point $x \in K$ and any $n\geq 0$, we have $\left\|\left.D f^n\right|_{F(x)}\right\| \leq C \lambda^n$;
        \item \textit{uniformly expanded} (by $D f$) if it is uniformly contracted by $D f^{-1}$.
     \end{enumerate}
      We say a compact invariant set $K$ is \textit{partially hyperbolic} if there is a $D f$-invariant dominated splitting $T_{K} M=E^u \oplus E_1^c \oplus E_2^c \oplus E^s$ such that
 \begin{itemize}
\item $E^u$ is uniformly expanded and $E^s$ is uniformly contracted.
\item  $E^u$ or $E^s$ may be trivial, but not simultaneously.
\end{itemize}
The bundles $E_1^c, E_2^c$ 
are called \textit{center bundles}. We write $E^c:=E_1^c\oplus E_2^c,$ and refer to $E^{cs}:=E^s\oplus E^c \quad\text{and}\quad E^{cu}:=E^u\oplus E^c $ as the \emph{center-stable} and \emph{center-unstable} subbundles, respectively.
\end{defn}

Bochi and Gourmelon \cite{BG} showed that a cocycle has \textit{dominated splitting} if and only if there is a uniform exponential gap between singular values of its iterates. Indeed, assuming that $X$ is a compact metric space, and letting $\mathcal{A}: X \to \glr$ be a matrix cocycle over a homeomorphism $(X,T)$. Let $Y\subset X$ be a non-empty $T$-invariant compact set. We say that the cocycle $\mathcal{A}$ has \textit{i-dominated splitting} on $Y$ if there are constants $C>0$ and $0<\tau<1$ such that
 \[ \frac{\sigma_{i+1}(\mathcal{A}^{n}(x))}{\sigma_{i}(\mathcal{A}^{n}(x))}\leq C \tau^n \hspace{0.2cm}\forall x\in Y, \forall n\in \N,\]
 where $\sigma_{i}$'s are singular values.

 Dominated splitting can be characterized in terms of the existence
of invariant cone fields (see \cite{CP}). We say that a matrix cocycle $\A$ has 1-dominated splitting on $Y$, which $Y\subset X$ is a non-empty $T$-invariant compact set, if there exist one-dimensional invariant cones $\mathcal{C}:=(C_x)_{x\in Y}$ such that
\[\mathcal{A}(x)C_{x} \subset C_{T(x)}^{o} \hspace{0.2cm}\forall x\in Y.\]
 If this is the case, $\mathcal{C}$ is called {\it unstable invariant cone field on $Y$}.

In this paper, we are interested in a $C^{1}$ diffeomorphism $f: M \rightarrow M$ on a compact Riemannian manifold that admits a dominated splitting $TM = E^{cs} \oplus E^{cu}$.
In \cite{Moh}, the author showed that if a cocycle has 1-dominated splitting  on some compact metric space, which can be characterized in terms of the existence of invariant cone fields (or multicones), then
 $ \|\mathcal{A}^{n}(x)\|$ is super-multiplicative for each $n$ and $x$.
\begin{prop}[{{\cite[Proposition 5.8]{Moh}}}] \label{prop:add}
Let $X$ be a compact metric space, and let $\mathcal{A}: X\rightarrow \glr$ be a matrix cocycle over a homeomorphism $(X, T)$. Assume that $(C_x)_{x\in X}$ is a 1-dimensional invariant cone field on $X$. Then, there exists $\kappa>0$ such that for every $m,n>0$ and for every $x\in X$ we have

\[
||\mathcal{A}^{m+n}(x)|| \geq \kappa ||\mathcal{A}^m(x)|| \cdot ||\mathcal{A}^n(T^m(x))||.
\]
\end{prop}


\begin{thm}\label{con1}
Assume that $(X, d)$ is a compact metric space. Let $\A:X \to \glr$ be a matrix cocycle over a homeomorphism $(X, T)$ and $\{\nu_{n}\}_{n=1}^{\infty}$ be a sequence in $\mathcal{M}(X)$. We form the new sequence $\{\mu_{n}\}_{n=1}^{\infty}$ by $\mu_{n}=\frac{1}{n}\sum_{i=0}^{n-1}T_{\ast}^{i}\nu_{n}$.  Assume that $\mu_{n_{i}}$ converges to $\mu$ in $\mathcal{M}(X)$ for some subsequence $\{n_i\}$ of natural numbers.  Then $\mu$ is an $T$-invariant measure and 
\[
\limsup_{i \rightarrow \infty} \frac{1}{n_{i}}\int \log  \|\A^{n_{i}}(x)\|d\nu_{n_i}(x)\leq \chi (\mu, \A).
\]
Moreover, if $\A$ has 1-dominated splitting on $X$, then

\[
\lim_{i \rightarrow \infty} \frac{1}{n_{i}}\int \log\|\A^{n_{i}}(x)\|d\nu_{n_i}(x)= \chi (\mu, \A).
\]
\end{thm}

\begin{proof}
The first part follows from \cite[Lemma A.2]{FH} and the second part follows from the combination of \cite[Lemma A.4]{FH} and Proposition \ref{prop:add}.
\end{proof}

We also have continuity of Lyapunov exponents when an invariant cone-field exists.

\begin{lem}\label{AAP}
Let $X$ be a compact metric space, and let $\mathcal{A}: X\rightarrow \glr$ be a matrix cocycle over a homeomorphism $(X, T)$. Assume that $\A$ has 1-dominated splitting on $X$. Then, for every $\epsilon >0$, there is $\delta>0$ such that if $\mu, \nu \in \mathcal{M}(X, T)$, and $d(\mu, \nu)<\delta$, then  $|\chi(\mu, \A)-\chi(\nu, \A)|<\epsilon.$
\end{lem}
\begin{proof}
It follows from the combination of Proposition \ref{prop:add} and \cite[Lemma  A.4]{FH}.
\end{proof}

\subsection{Ergodic optimization}
Assume that $(X, d)$ is a compact metric space. Let $\A:X \to \glr$ be a matrix cocycle over a homeomorphism $(X,T).$ {\it Ergodic optimization of Lyapunov exponents} concerns the supremum of the Lyapunov exponents of measures over invariant measures:
\begin{equation}\label{max2}
\beta(\A):= \sup_{\mu \in \mathcal{M}(X,T)}\chi(\mu, \A).
\end{equation}
In \eqref{max2}, the supremum is always attained by an ergodic measure; such measures will be called {\it Lyapunov maximizing measures}. This follows from the fact that $\mathcal{M}(X, T)$ is a compact convex set  whose extreme points are exactly the ergodic measure, and $\chi(\cdot,\A)$ is upper semi-continuous with respect to the weak* topology (see \cite{jen19, Bochi-ICM, Moh-25}).

Let
us stress that the maximal Lyapunov exponent can also be characterized as follows:
\begin{equation}\label{maximal LE}
\beta(\A)=\limsup_{n\to \infty}\frac{1}{n}\sup_{x\in X}\log \|\A^{n}(x)\|=\sup_{x\in X}\limsup_{n\to \infty}\frac{1}{n}\log \|\A^{n}(x)\|.
\end{equation}

We finish this section with the following instrumental result.
\begin{lem}\label{one_side}
Assume that $X$ is a compact manifold. Let $\A:X \to \glr$ be a matrix cocycle over a homeomorphism $(X,T)$. Then
$$
\begin{aligned}
\ess \sup \limsup_{n\rightarrow \infty} \frac{1}{n}\log \|\A^n(x)\|& \leq \sup_{\mu \in \mathcal{O}_T} \chi(\mu, \A)\\& \leq \limsup_{n\rightarrow \infty}\frac{1}{n}\ess \sup \log \|\A^{n}(x)\|,
\end{aligned}
$$
where the essential supremum is taken with respect to $\Leb.$ 
\end{lem}
\begin{proof}

Choose a point $x\in \mathcal{G}$ (see \eqref{basin-of-attraction}), and then a subsequence of integers $\{n_{i}\}$ such that $\limsup_{n\rightarrow \infty}\frac{1}{n}\log \|\A^{n}(x)\|=\lim_{i\rightarrow \infty}\frac{1}{n_{i}}\log \|\A^{n_i}(x)\|.$ Since $x \in \mathcal{G}$, there is a subsequence $\gamma_{n_{i_{k}}, x}$ of $\gamma_{n_{i}, x}$ (see \eqref{empirical}) such that converges in weak$^{\ast}$ topology to an observable measure $\mu$ when $k\rightarrow \infty$ by definition. Without loss of generality, we can suppose that $\gamma_{n_{i}, x}$ converges to $\mu$ as $i  \to \infty$. Then, by Theorem \ref{con1},
\begin{equation}\label{com1}
 \lim_{i\rightarrow \infty}\frac{1}{n_{i}} \int \log \|\A^{n_i}(x)\|d\delta_{x}\leq  \chi(\mu, \A)\leq \sup_{\mu \in \mathcal{O}_T}\chi(\mu, \A).
\end{equation}
On the other hand,
\begin{equation}\label{com2}
 \lim_{i\rightarrow \infty}\frac{1}{n_{i}} \int \log \|\A^{n_i}(x)\| d\delta_{x}=\limsup_{n\rightarrow \infty}\frac{1}{n} \log \|\A^{n}(x)\|.
\end{equation}

Combination (\ref{com1}) and (\ref{com2}) and the arbitrariness of $x \in \mathcal{G}$ imply 
\[ \ess \sup \limsup_{n\rightarrow \infty} \frac{1}{n}\log \|\A^{n}(x)\| \leq \sup_{\mu \in \mathcal{O}_T} \chi(\mu, \A).\]

We consider any subset $U$ with full Lebesgue measure. We know that each subset of full Lebesgue measure is dense, hence
 \[ \limsup_{n\rightarrow \infty}\frac{1}{n}\sup_{x \in U} \log \|\A^{n}(x)\|=\limsup_{n\rightarrow \infty}\frac{1}{n} \sup_{x\in X} \log \|\A^{n}(x)\|=\beta(\A).\]
 
By \eqref{maximal LE} and \eqref{max2},
 \[ \ess \sup \limsup_{n\rightarrow \infty}\frac{1}{n} \log \|\A^{n}(x)\| \leq \sup_{\mu \in \mathcal{O}_T} \chi(\mu, \A) \leq \limsup_{n\rightarrow \infty}\frac{1}{n}\ess \sup \log \|\A^{n}(x)\|.\]
\end{proof}

\begin{proof}[Proof of Corollary~\ref{corCinfty}]

We define $
\phi_n(x) := \max_{k} \|\wedge^k Df_x^n\|$. By Kozlovski's integral formula \eqref{definition of entropy} , the topological entropy of $f$ satisfies
\[
h_{\mathrm{top}}(f) = \lim_{n \to \infty} \frac{1}{n} \log \int \phi_n(x)\, d\operatorname{Leb}(x).
\]

\medskip
For each $n$,
\[
\int \phi_n(x)\, d\operatorname{Leb}(x)\le \ess \sup \phi_n(x).
\]
Taking $\frac{1}{n} \log$ and $\limsup_{n \to \infty}$ yields
\begin{equation}\label{upper bound-entropy}
h_{\mathrm{top}}(f) = \lim_{n \to \infty} \frac{1}{n} \log \int \phi_n(x)\, d\operatorname{Leb}(x)
\le \limsup_{n \to \infty} \frac{1}{n} \log \ess \sup \phi_n(x).
\end{equation}

\medskip
We denote by
\[
\lambda(x) := \limsup_{n \to \infty} \frac{1}{n} \log \phi_n(x),
\qquad
A := \operatorname{ess\,sup} \lambda(x).
\]

For any $\varepsilon > 0$, we define
\[
E_\varepsilon := \{ x \in M : \lambda(x) > A - \varepsilon \}.
\]
Then, by the definition of the essential supremum, we have $\operatorname{Leb}(E_\varepsilon) > 0.$ 
For each $n \in \mathbb{N}$, let
\[
A_n := \left\{ x \in M : \frac{1}{n} \log \phi_n(x) > A - \varepsilon \right\}.
\]
If $x \in E_{\varepsilon}$, then $
\limsup _{n \rightarrow \infty} \frac{1}{n} \log \phi_n(x)>A-\varepsilon,
$ hence $x \in A_n$ for infinitely many $n$. Therefore
 $E_\varepsilon \subset \limsup_{n\to\infty} A_n=\bigcap_{m=1}^{\infty} \bigcup_{n \ge m} A_n$.  Consequently, for every $m$,
\[
E_\varepsilon \subset \bigcup_{n\ge m} A_n
\]
hence
\[
E_\varepsilon \subset \bigcup_{n\ge m} (E_\varepsilon \cap A_n).
\]
As $\operatorname{Leb}(E_\varepsilon)>0$, we claim that for every $\delta>0$, there exist infinitely many $n$ such that
\[
\operatorname{Leb}(E_\varepsilon\cap A_n)\ge e^{-\delta n}.
\]

Indeed, suppose by contradiction that there exist $\delta>0$ and $N\in\mathbb N$
such that
\[
\operatorname{Leb}(E_\varepsilon\cap A_n)< e^{-\delta n}
\qquad\forall n\ge N.
\]
Then, for every $m\ge N$,
\[
\operatorname{Leb}\!\left(\bigcup_{n\ge m}(E_\varepsilon\cap A_n)\right)
\le
\sum_{n\ge m}\operatorname{Leb}(E_\varepsilon\cap A_n)
\le
\sum_{n\ge m} e^{-\delta n}.
\]
Since the geometric series converges, the right-hand side tends to $0$ as
$m\to\infty$. On the other hand,
\[
E_\varepsilon \subset \bigcup_{n\ge m}(E_\varepsilon\cap A_n)
\qquad\forall m,
\]
hence
\[
0<\operatorname{Leb}(E_\varepsilon)
\le
\operatorname{Leb}\!\left(\bigcup_{n\ge m}(E_\varepsilon\cap A_n)\right),
\]
which is a contradiction.

Therefore, for every $\delta>0$, there exists a sequence $n_k\to\infty$ such that
\[
\operatorname{Leb}(E_\varepsilon\cap A_{n_k})\ge e^{-\delta n_k}.
\]

For such $n_k$,
\[
\int \phi_{n_k}\, d\operatorname{Leb}
\ge
\int_{E_\varepsilon \cap A_{n_k}} \phi_{n_k}\, d\operatorname{Leb}
\ge
\operatorname{Leb}(E_\varepsilon \cap A_{n_k})\, e^{n_k(A-\varepsilon)}
\ge
e^{n_k(A-\varepsilon-\delta)}.
\]
Hence
\[
\frac{1}{n_k}\log \int \phi_{n_k}\, d\operatorname{Leb}
\ge
A-\varepsilon-\delta.
\]
Taking the $\limsup$ yields
\[
\limsup_{n\to\infty} \frac{1}{n}\log \int \phi_n\, d\operatorname{Leb}
\ge
A-\varepsilon-\delta.
\]
Since $\varepsilon>0$ and $\delta>0$ are arbitrary,
\begin{equation}\label{lower-bound-entropy}
    \limsup_{n \to \infty} \frac{1}{n} \log \int \phi_n(x)\, d\operatorname{Leb}(x)
\ge A.
\end{equation}

 Combining \eqref{upper bound-entropy} and \eqref{lower-bound-entropy}, we conclude
\[
\ess \sup \limsup_{n \to \infty} \frac{1}{n} \log \phi_n(x) 
\;\le\; h_{\mathrm{top}}(f) 
\;\le\; \limsup_{n \to \infty} \frac{1}{n} \log \ess \sup \phi_n(x).
\]

 \end{proof}
\color{black}

\section{Proofs}\label{proofs}

\subsection{Proof of Theorem \ref{mainresult3}}\label{SRB0}

We recall that $(X,T)$ is a homeomorphism on a compact manifold $X$ and that $\Leb$ is a Lebesgue measure normalized so that $\Leb(X)= 1.$  Let $\A: X \to \glr$ be a matrix cocycle over $(X, T)$.

As $\chi(., \A)$ is upper semi-continuous with respect to weak* topology and  $\mathcal{O}_T$ is weak* compact by Theorem \ref{total}, we have
\begin{equation}\label{max obser measure}
    \sup_{\nu \in \mathcal{O}_T}\chi(\nu, \A)=\max_{\nu \in \mathcal{O}_T}\chi(\nu, \A).
\end{equation}

Similar to Lyapunov maximizing measures, we define maximizing observable measures as follows:
\begin{defn}\label{def_LMO}
An observable measure $\mu$ is called a {\it maximizing observable measure} for the matrix cocycle $\A:X \to \glr$ over $(X, T)$ if \[
\chi(\mu, \A)=  \sup_{\nu \in \mathcal{O}_T}\chi(\nu, \A)=\max_{\nu \in \mathcal{O}_T}\chi(\nu, \A).
\]
\end{defn}

By \eqref{max obser measure}, a maximizing observable measure always exists. Similar to \eqref{maximal LE} and \eqref{max2}, we obtain the following theorem for observable measures.
 \begin{thm}\label{main-conj}
 Let $\mathcal{A}:X \to \glr$ be a matrix cocycle over a homeomorphism $(X,T)$. Let $\mu$ be an ergodic maximizing observable measure. Suppose that the cocycle $\mathcal{A}$ has a 1-dominated splitting on the support $\mu$. Then
\begin{equation}\label{mainn}
\ess \sup \limsup_{n\rightarrow \infty} \frac{1}{n}\log \|\mathcal{A}^{n}(x)\|=\sup_{\mu \in \mathcal{O}_T}\chi(\mu, \A)
\end{equation}
\begin{equation}\label{ineq for the LMO}
\hspace{5.9cm}\leq \limsup_{n\rightarrow \infty}\frac{1}{n}\ess \sup \log \|\mathcal{A}^{n}(x)\|,
\end{equation}
where the $\ess \sup$ denotes the essential supremum taken with respect to the Lebesgue measure $\Leb.$
\end{thm}

\color{black}
\begin{rem}

The inequality in the
formula \eqref{ineq for the LMO} could be strict (see Example \ref{res}).
\end{rem}

We are now in a position to prove Theorem \ref{mainresult3}. In fact, since this is a consequence of Theorem \ref{main-conj}, we shall prove it later.

Assume that $\mu$ is an invariant measure which may not be ergodic. We consider its ergodic decomposition. By the affine property of the Lyapunov exponent, the Lyapunov exponent of $\mu$ equals the weighted average of the exponents of its ergodic components (see \cite[Proposition A.1 (3)]{FH}). Therefore, there exists at least one ergodic component whose Lyapunov exponent equals the Lyapunov exponent of $\mu$. One may use this fact in the proof below.
\color{black}

\begin{proof}[Proof of Theorem \ref{main-conj}]
    
\color{black}
Lemma \ref{one_side} implies that one side of the equality \eqref{mainn} and the inequality \eqref{ineq for the LMO} hold.
Now, we explain how one can tackle the other side of the equality \eqref{mainn} in Theorem \ref{main-conj}.

 Let $\mathcal{G} \subset X$ such that $V(x) \subseteq \mathcal{O}_T$ for any $x \in \mathcal{G}$. By Theorem \ref{total}, $\operatorname{Leb}(\mathcal{G})=1$. We fix a subset $L\subset X$ with the full Lebesgue measure.

The set of observable measures $\mathcal{O}_T$ is weak* compact (see Theorem \ref{total}), hence there is a maximizing observable measure $\mu$ for the matrix cocycle $\A$. Let $\mu$ be an ergodic maximizing observable measure. Moreover, we recall that we assume the cocycle $\mathcal{A}$ has a 1-dominated splitting  on $Y:=\text{supp}(\mu)$ (the support of $\mu$). By the robustness of dominated splittings (see \cite[Corollary 2.8]{CP}), there exists a closed neighborhood $U$ of $\text{supp}(\mu)$ such that  $Z:=\bigcap_{k \in \Z} T^{-k}(U) \supseteq \text{supp}(\mu)$ is the maximal invariant set in this neighborhood. Then, the restricted bundle $E_Z$ over the compact invariant set $Z$ admits a dominated splitting $\mathbb{F} \oplus \mathbb{G}$ such that $\dim(\mathbb{F})=1$, extending the previously found dominated splitting on $E_Y$.  Therefore, the cocycle $\A$  has 1-dominated splitting  on $Z.$

We choose $\delta$ small enough such that $B_{\delta}(\mu) \subset Z$ (see Remark \ref{approaching}). Note that the basin of $\delta$-attraction of $\mu$, $B_{\delta}(\mu)$, has positive Lebesgue measure by the definition of the observable measure.

 Dominated splitting can be characterized in terms of the existence
of invariant cone fields. Therefore, there is a 1-dimensional family of invariant cones $(C_{x})_{x\in Z}$ such that 
\[\A(x)(C_{x}) \subset C_{T(x)}^{o}.\]
 Hence, $\mu \to \chi(\mu, \A)$ is continuous on $\mathcal{M}(Z, T_{|Z})$ by Lemma \ref{AAP}.

There is a compact set $W$ with $\Leb(W)>0$ such that $W \subset B_{\delta}(\mu)$. Let $y \in W \subset B_{\delta}(\mu) \subset Z \cap L \cap \mathcal{G}$. Hence, there is a subsequence of integers $\{n_{i}\}$ such that $\frac{1}{n_{i}}\sum_{j=0}^{n_{i}-1}\delta_{T^{j}(y)}$ converges to a measure $\nu$ with $d(\mu, \nu)<\delta$ (distance in weak* topology) by definition.

 Now, we prove the other side of the equality \eqref{mainn}.
$$ \begin{aligned}
\limsup_{n\rightarrow \infty}\frac{1}{n}\log \|\mathcal{A}^{n}(y)\| &\geq \limsup_{i\rightarrow \infty}\frac{1}{n_{i}}\log \|\mathcal{A}^{n_{i}}(y)\|\\&=\underbrace{\limsup_{i\rightarrow \infty} \frac{1}{n_{i}}\int \log \|\mathcal{A}^{n_{i}}(y)\| d\delta_{y}}_{\text{\circled{1}}}.
\end{aligned}
$$

Since $\frac{1}{n_{i}}\sum_{j=0}^{n_{i}-1}\delta_{T^{j}(y)} \rightarrow \nu$ (in weak*topology) and there is the  1-dimensional family of invariant cones $(C_{x})_{x\in Z}$, by Theorem \ref{con1},
\[
\circled{1}=\underbrace{\chi(\nu, \A)}_{\text{\circled{2}}}.
\]
By continuity (Lemma \ref{AAP}), for given $\epsilon>0,$
\[\circled{2} \geq \chi(\mu, \A)-\epsilon.\]

That implies
\[ \sup_{x\in L} \limsup_{n\rightarrow \infty} \frac{1}{n}\log \|\mathcal{A}^{n}(x)\| \geq \chi(\mu, \A)- \epsilon.\]

As we consider arbitrary $L$ and $\epsilon,$
\begin{equation}\label{one-side-proof}
\ess \sup \limsup_{n\rightarrow \infty} \frac{1}{n}\log \|\mathcal{A}^{n}(x)\| \geq \sup_{\mu \in \mathcal{O}_T}\chi(\mu, \A).
\end{equation} 
Combining \eqref{one-side-proof} and Lemma \ref{one_side} completes the proof of the theorem.
\end{proof}
\color{black}
\begin{cor}\label{Obs=phy}
Under conditions of Theorem \ref{main-conj}, we assume that the set of observable measures $\mathcal{O}_T$ is at most countable, then we have
\[\ess \sup \limsup_{n\rightarrow \infty} \frac{1}{n}\log \|\mathcal{A}^{n}(x)\|=\sup_{\mu \in \mathcal{O}_T}\chi(\mu, \A)=\sup_{\mu \in \textbf{Phys}}\chi(\mu, \A).\]
\end{cor}
\begin{proof}
It is enough to prove that \[\sup_{\mu \in \mathcal{O}_T}\chi(\mu, \A)=\sup_{\mu \in \textbf{Phys}}\chi(\mu, \A),\]
by Theorem \ref{main-conj}.

If the set of observable measures $\mathcal{O}_T$ is finite, then $\mathcal{O}_T=\textbf{Phys}$ by Theorem \ref{total}, 
and if $\mathcal{O}$ is countable infinite, then $\mathcal{O}_T$ is the weak*closure of the set of $\textbf{Phys}$ measures by Theorem \ref{total}. In both cases, the equality is clear.
\end{proof}
\begin{rem}
The physical measures may not exist if the set $\mathcal{O}_T$ is uncountable. For instance,
the identity map on a manifold has observable measures but has no physical measures. Hence, the
above result fails in general when the set $\mathcal{O}$ is uncountable.
\end{rem}

\begin{proof}[Proof of Theorem \ref{mainresult3}]
The proof follows from Theorem \ref{main-conj} by considering $Df^{-1}_{|F(x)}=\A(x)$.
\end{proof}

The next example shows that the inequality \eqref{ineq for the LMO} can be strict.

\begin{ex}\label{res}
Let $f$ be a transitive $C^{1+\alpha}$ Anosov diffeomorphism
and $\mu$ be the unique observable measure, which is the 
 unique SRB measure. There are infinitely many other
ergodic measures, which are not observable (e.g. those supported on the periodic orbits). 
Hence, by \cite[Theorem 1]{J}, one can find a continuous function whose integral with respect to
some ergodic measure is larger than its integral with respect to any observable measure $\mu$.
\end{ex}

\color{black}

\subsection{Proof of Theorem \ref{mainresult1}}\label{SRB}
Now we are in a position to prove Theorem \ref{mainresult1}.

The starting point in the proof is the following result by
Tian \cite{Ti}, which shows that the negative Lyapunov exponents of all observable measures imply non-uniform expansion.
\begin{thm}[{{\cite[Proposition A]{Ti}}}]\label{eqi:NUE}Let $T:X\to X$ be a continuous map on a compact manifold $X$. Assume that $a_n: X \to \R$ and
$b_n: X\to \R$, $n\geq 0$, are two sequences of continuous functions such that for every $x\in X$,$ n,k \geq 0$,
\begin{equation}\label{subadditive}
a_{n+k}(x)\leq a_n(T^k(x))+ a_k(x),\quad  b_{n+k}(x) \leq b_n (T^k(x))+ b_k (x);
\end{equation}
and
\begin{equation}\label{subadditive1}
a_n (x)\leq a_{n+k}(x) + b_k (T^n(x)).
\end{equation}

If 
\[ \lim_{n\to \infty} \frac{1}{n}\int a_{n}(x)d\mu(x)<0\]
for every $\mu \in \mathcal{O}_T$, then there are $\lambda >0$ and $K\in \N$ such that Lebesgue almost every $x$
satisfies
\[\limsup_{n \to \infty}\frac{1}{nK}\sum_{i=0}^{n-1}a_{K}(T^{iK}(x))\leq -\lambda<0.\]
\end{thm}

\begin{proof}[Proof of Theorem \ref{mainresult1}]
    
\color{black}

Note that condition \eqref{LE00} is equivalent to 
\begin{equation}\label{equal1}
\ess \sup \limsup_{n\to \infty}\frac{1}{n}\log\|Df^{-n}_{|E^{cu}(f^{n}(x))}\|<0.
\end{equation}
On the other hand,
\begin{equation}\label{optimization for PHS over observable}
    \ess \sup \limsup_{n\to \infty}\frac{1}{n}\log\|Df^{-n}_{|E^{cu}(f^{n}(x))}\|=\max_{\mu \in \mathcal{O}_f}\lambda^{-}(\mu, f),
\end{equation}
by Theorem \ref{mainresult3}. Therefore, by \eqref{equal1} and \eqref{optimization for PHS over observable},
\begin{equation}\label{equal11}
\lambda^{-}(\mu, f) <0
\end{equation}
for any $\mu \in \mathcal{O}_f.$
We denote $a_{n}(x):=\log \|Df_{|_{E^{cu}(f^{n}(x))}}^{-n}\|$ and $b_n(x):=\log \|Df_{|E^{cu}(x)}^{n}\|$. They satisfy conditions \eqref{subadditive} and \eqref{subadditive1} of Theorem \ref{eqi:NUE}. Moreover, Lyapunov exponents of all observable measures are negative (see \eqref{equal11}). Then there are $\lambda>0$ and $K \in \mathbb{N}$ such that Lebesgue a.e. $x$ satisfies
\[
\limsup _{n \rightarrow+\infty} \frac{1}{n K} \sum_{i=1}^{n} \log \left\|D f_{|E^{cu}(f^{i K}(x))}^{-K}\right\| \leq-\lambda<0,
\]
 by Theorem \ref{eqi:NUE}.
This completes the proof of the first part of theorem. The second part of the proof follows from the combination of the first part of the Theorem and  \cite[Theorem A]{ADLP}.

\end{proof}

\subsection{Proof of Corollary~\ref{equal O=SRB}}

We aim to show that 
\[
\mathcal{O}_f = \mathbf{Phys} = \mathbf{SRB}.
\]

By Theorem~\ref{mainresult1}, there exist finitely many ergodic SRB measures for $f^K$, say $\mu_1, \ldots, \mu_k$, whose basins cover a set of full Lebesgue measure, where $K \ge 1$ is the integer given by Theorem~\ref{mainresult1}. 

To obtain measures that are invariant under $f$, we define, for each $i$ and any Borel set $B$,
\[
\nu_i(B) := \frac{1}{K} \sum_{j=0}^{K-1} \mu_i(f^{-j}(B)).
\]

Then each $\nu_i$ is $f$-invariant, ergodic, and an SRB measure for $f$. Note that if $\mu_i$ is already $f$-invariant, then the formula gives $\nu_i = \mu_i$. In general, it may happen that $\nu_i = \nu_j$ for some $i \neq j$. The set $\{\nu_1, \dots, \nu_k\}$ forms a compact subset of $\mathcal{O}_f$, and the union of their basins still has full Lebesgue measure.

Finally, by Theorem~\ref{total}, $\mathcal{O}_f$ coincides with these finitely many measures $\nu_i$. The result then follows from Corollary~\ref{Obs=phy}.
\color{black}

\medskip
 \textbf{\textit{Acknowledgments.}}
The author is grateful to Jos{\'e} Alves, Davi Obata, David Burguet, Snir Ben Ovadia, and Kiho Park for useful discussions. He also thanks Paulo Varandas for his invaluable assistance in preparing the introduction of the paper. This work was supported by the Knut and Alice Wallenberg Foundation.


\bibliographystyle{acm}
\bibliography{SRB-new}
\end{document}